\documentclass[twoside, 12pt]{article}

\usepackage[english]{babel}
\usepackage{epsfig}
\usepackage{amssymb,amsmath,amsfonts,soul,amsthm,enumerate}
\usepackage{graphicx}
\usepackage[utf8]{inputenc}
\usepackage[T2A]{fontenc}
\usepackage{mathrsfs}
 \newtheorem{theorem}{Theorem}[section]
 \newtheorem{corollary}[theorem]{Corollary}

 \theoremstyle{definition}
 \newtheorem{definition}[theorem]{Definition}
 \theoremstyle{remark}
 \newtheorem{remark}[theorem]{Remark}
 
 \numberwithin{equation}{section}

 \def\@evenhead{\vbox{\hbox to \textwidth{\thepage\hfil\sl\leftmark\strut}\hrule}}
 \def\@oddhead{\vbox{\hbox to \textwidth{\rightmark\hfill\thepage\strut}\hrule}}

\usepackage{hyperref}
\usepackage{alphalph}
\usepackage[subnum]{cases}
\usepackage{array}

\title{\bf  Non-Volterra property of some class of compact operators}
\author{B.N. Biyarov}

\begin{document}
 \sloppy
\maketitle


\markboth{\hfill{\footnotesize\rm   B.N.~Biyarov}\hfill}
{\hfill{\footnotesize\sl   Non-Volterra property of some class of compact operators}\hfill}
\vskip 0.3cm  
\vskip 0.7 cm


\vskip 0.2cm

\vskip 0.2cm

%
\noindent {\bf Abstract.} The authors Matsaev and Mogulskii singled out a wide class of weak perturbation of a positive compact operator $H$, of the form $H(I+S)$, where $S$ is such a compact operator that $I+S$ is continuously invertible, which does not have a nonzero eigenvalue, i.e., is Volterra. On the other hand, such weak perturbations have a complete system of root vectors if the self-adjoint operator $H$ is from the Schatten-von Neumann class. In this paper, we consider the compact operators $A$ representable as a sum of two compact operators $A=C+T$, i.e., $A$ is not necessarily a weak perturbation, where $C$ is a non-negative operator. We will prove existence theorems of nonzero eigenvalues for such operators. The study of the spectral properties of operators generated by differential equations  with Cauchy initial data involve, as a rule, Volterra boundary-value problems that are well posed. But Hadamard's example shows that the Cauchy problem for the Laplace equation is ill posed. At present, not a single Volterra well-defined restriction or extension for elliptic-type equations is known. Thus, the following question arises: Does there exist a Volterra well-defined restriction of a maximal operator $\widehat{L}$ or a Volterra well-defined extension of a minimal operator $L_0$ generated by the Laplace operator? The obtained existence theorems for eigenvalues give that a wide class of well-defined restrictions of the maximal operator $\widehat{L}$ and a wide class of well-defined extensions of the minimal operator $L_0$ generated by the Laplace operator cannot be Volterra. Moreover, in the two-dimensional case it is proven that there are no Volterra well-defined restrictions or extensions for Laplace operator at all.


\section{\large Introduction}
\label{section1}
Let us present some definitions, notation, and terminology.

In a Hilbert space $H$, we consider a linear operator $L$ with  domain $D(L)$ and range $R(L)$. 
By the \textit{kernel} of the operator $L$ we mean the set
\[\mbox{Ker}\,L=\bigl\{f\in D(L): \; Lf=0\bigl\}.\]

\begin{definition}
\label{def1}  
An operator $L$ is called a \textit{restriction} of an operator $L_1$, and $L_1$ is called an \textit{extension} of  $L$  if

1) $D(L)\subset D(L_1)$;

2) $Lf=L_1f$ for all $f$ from $D(L)$.
\\
Usually, one writes $L\subset L_1$.
\end{definition}

\begin{definition}
\label{def2}
A linear closed operator $L_0$ in a Hilbert space $H$ is called \textit{minimal} if there exists a bounded inverse operator $L_0^{-1}$ on $R(L_0)$ and ${R(L_0)}\not=H$; a linear closed operator $\widehat{L}$ in a Hilbert space $H$ is called \textit{maximal} if $R(\widehat{L})=H$ and $\mbox{Ker}\, \widehat{L} \not=\{0\}$; a linear closed operator $L$ in a Hilbert space $H$ is called \textit{well defined} if there exists a bounded inverse operator $L^{-1}$ defined on all of $H$. 
\end{definition}

We say that a well defined operator $L$ in a Hilbert space $H$ is a \textit{well-defined extension} of a minimal operator $L_0$ (\textit{well-defined restriction} of a maximal operator $\widehat{L}$) if $L_0\subset L$ ($L\subset \widehat{L}$).
A well defined operator $L$ in a Hilbert space $H$ is a \textit{boundary well-defined extension} of a minimal operator $L_0$ with respect to a maximal operator $\widehat{L}$ if $L$ is both a well-defined restriction of the maximal operator $\widehat{L}$ and a well-defined extension of the minimal operator $L_0$, i.e., $L_0\subset L \subset \widehat{L}$.

At the beginning of the 1950s, Vishik {\cite{Vishik}} extended the theory of self-adjoint extensions of von Neumann--Krein symmetric operators to nonsymmetric operators in Hilbert space.

Later, Bitsadze and Samarskii \cite{Bitsadze} discovered a correct problem not contained among the problems described by Vishik. For ordinary differential equations, problems of such type were studied by Dezin \cite{Dezin}.

At the beginning of the 1980s, Otelbaev and his disciples proved abstract theorems whose application makes it possible to describe all well-defined extensions of some minimal operator using any single known well-defined extension in terms of an inverse operator. Here such extensions need not be restrictions of a maximal operator. Similarly, all possible well-defined restrictions of some maximal operator that need not be extensions of a minimal operator were described (see {\cite{Kokebaev}}). This description covers problems of Bitsadze-Samarskii type.

Suppose that $\widehat{L}$ is a maximal linear operator in a Hilbert space $H$,  $L$ is any known well-defined restriction of $\widehat{L}$, and $K$ is an arbitrary linear bounded (in $H$) operator satisfying the following condition:
\begin{equation}\label{1.1}
R(K)\subset \mbox{Ker}\, \widehat{L}.
\end{equation}
Then the operator $L_K^{-1}$ defined by the formula 
\begin{equation}\label{1.2}
L_K^{-1}f=L^{-1}f+Kf,
\end{equation}
describes the inverse operators to all possible well-defined restrictions $L_K$ of $\widehat{L}$, i.e., $L_K\subset \widehat{L}$.

Suppose that  $L_0$ is a minimal operator in a Hilbert space $H$,  $L$ is any known well-defined extension of $L_0$, and $K$ is a linear bounded operator in $H$ satisfying the conditions

a) $R(L_0)\subset \mbox{Ker}\,K$,

b) $\mbox{Ker}\,(I+KL)=\{0\}$,
\\
then the operator $L_K^{-1}$ defined by formula \eqref{1.2}
describes the inverse operators to all possible well-defined extensions $L_K$ of  $L_0$,  i.e., $L_0\subset L_K$.

Let $L$ be any known boundary well-defined extension of $L_0$, i.e., $L_0\subset L\subset \widehat{L}$. The existence of at least one boundary well-defined extension $L$ was proved by Vishik in {\cite{Vishik}}. Let $K$ be a linear bounded (in $H$) operator satisfying the conditions

a) $R(L_0)\subset \mbox{Ker}\,K$,

b) $R(K)\subset \mbox{Ker}\,\widehat{L}$,
\\
then the operator $L_K^{-1}$ defined by formula \eqref{1.2}
describes the inverse operators to all possible boundary well-defined extensions $L_K$ of $L_0$,  i.e., $L_0\subset L_K\subset \widehat{L}$. 

\begin{definition}
\label{def7} A bounded operator $A$ in a Hilbert space $H$ is called \textit{quasinilpotent} if its spectral
radius is zero, that is, the spectrum consists of the single point zero.
\end{definition}

\begin{definition}
\label{def8} An operator $A$ in a Hilbert space $H$ is called a \textit{Volterra operator} if $A$ is compact and quasinilpotent.
\end{definition}

\begin{definition}
\label{def9} A well-defined operator $L$  is said to be \textit{Volterra} if the inverse operator $L^{-1}$ is a
Volterra operator.
\end{definition}

Denote by $\mathfrak{S}_{\infty}\left(H,H_{1}\right)$ the set of
all linear compact operators from a Hilbert space $H$ to a Hilbert
space $H_{1}$. If $T\in\mathfrak{S}_{\infty}\left(H,H_{1}\right)$,
then $T^{*}T$ is a nonnegative self-adjoint operator from the set
$\mathfrak{S}_{\infty}\left(H\right)\equiv\mathfrak{S}_{\infty}\left(H,H\right)$
and, moreover, there is a nonnegative unique self-adjoint root $\left|T\right|=\left(T^{*}T\right)^{1/2}$
in $\mathfrak{S}_{\infty}\left(H\right)$. The eigenvalues $\lambda_{n}\left(\left|T\right|\right)$
constitute a monotone sequence of nonnegative numbers converging to
zero. These numbers are usually called $s$-\textit{numbers} (see
\cite{Gohberg}) of the operator $T$ and denoted by $s_{n}\left(T\right)$,
$n\in\mathbb{N}$. By $\mathfrak{S}_{p}\left(H,H_{1}\right)$ we denote
the set of all compact operators $T\in\mathfrak{S}_{\infty}\left(H,H_{1}\right)$
satisfying the condition
\[
\sum_{j=1}^{\infty}s_{j}^{p}\left(T\right)<\infty,\quad0<p<\infty.
\]

Obviously, if rank $|T| = r < \infty$, then $s_n(T) = 0$, for $n = r + 1, r + 2,... $. Operators of finite rank certainly belong to the classes $\mathfrak{S}_{p}\left(H,H_{1}\right)$ for all $p > 0$.


\section{\large Main Result}
\label{section2}\label{sec2}

In the work \cite{Matsaev}, a wide class of weak perturbation of a positive compact operator $H$, of the form $H(I+S)$, was singled out, where $S$ is such a compact operator that $I+S$ is continuously invertible, which does not have a nonzero eigenvalue, i.e., is Volterra.
On the other hand, such weak perturbations have a complete system of root vectors if the self-adjoint operator $H$ is from the Schatten-von Neumann class (see Theorem 8. 1 \cite[p.257]{Gohberg}).
In this paper, we consider compact operators $A$ representable as a sum of two compact operators $A=C+T$, where $C$ is a non-negative operator.
In this section \ref{sec2}, we will prove existence theorems for nonzero eigenvalues for such operators.

\begin{theorem}\label{Th2.1}
If the  operator $A$ in a Hilbert space $H$ can be represented in the form
$$A=C+T,$$
where $C$ is a non-negative  operator such that
\[
\lim_{n\rightarrow\infty} n^2 s_n(C)=\infty, \quad \lim_{n\rightarrow\infty} n^{\alpha} s_n(C)=a, \quad \mbox{for} \quad 0<\alpha<2, \;\;\; 0< a<\infty,
\]
and the operator $T\in S_q(H)$ for all $q\leq\frac{1}{\alpha}$, then the operator $A$ cannot be  Volterra.
\end{theorem}

\begin{proof}
It is clear from the condition of Theorem \ref{Th2.1} that $s_n(C)$ and $s_n(T)$ tend to zero as $n$ tends to infinity, i.e., they are compact operators.
It is known (see \cite[p.95]{Gohberg}) that
\[
\lim_{n\rightarrow\infty} n^{\frac{1}{q}} s_n(T)=0.
\]
Notice that
\[
\lim_{n\rightarrow\infty} n^{\alpha} s_n(T)=\lim_{n\rightarrow\infty} \bigg(\frac{n^{\alpha}}{n^{\frac{1}{q}}}\bigg) n^{\frac{1}{q}} s_n(T)=0,
\]
since $\alpha\leq\frac{1}{q}$. Then by Theorem 2.3 \cite[p.32]{Gohberg}
\[
\lim_{n\rightarrow\infty} n^{\alpha} s_n(A)=a.
\]

To prove Theorem \ref{Th2.1} assume the contrary, i.e., suppose that the operator $A$ is  Volterra.

Case 1. Let the given parameter $\alpha$ be in the interval $(0, 1)$. Then $1<\frac{1}{\alpha}<\infty$ and $0<q\leq\frac{1}{\alpha}<\infty$. It is clear that
\[
T_J=\frac{T-T^*}{2i}\neq 0.
\]
Otherwise, $A$ is self-adjoint and contradicts the Volterra property of $A$.
 
a) If $0<q\leq 1/2$ then by Theorem 11.5  \cite[p.220]{Gohberg} we get that 
$$\lim_{n\rightarrow\infty} n^2 s_n(C)$$ 
exists and is finite.
This contradicts the condition of Theorem \ref{Th2.1}. 

b) Consider the case $1/2<q\leq1$. Then by Theorem 11.3 \cite[p.217]{Gohberg} we have that $C\in S_q$ for all $1/2<q\leq1$, i.e., $\rho(C)= 1/2$.
But we consider the case $0<\alpha<1$, i.e., $C\in S_p$, for all $p>1/\alpha>1$,i.e, $\rho(C)=1/\alpha> 1,$. This is a contradiction to the fact that $\rho(C)= 1/2$.

c) Now we consider the case $1<q\leq1/\alpha<\infty$. 
Despite the fact that $T\in S_q, \; q>1$ it may turn out that $T_J\in S_q,\; q \leq 1$.
Then by Theorem 2.1  \cite[p.83]{Gohberg2}  we have that $A\in S_{\Omega}$ which is of order $\rho(A)=\inf \{p:\; A\in S_{\Omega}\}\leq 1$. But we have $A_R=C+T_R$.
By the condition of Theorem \ref{Th2.1},  
$$\rho(C)=\inf \{p:\; C\in S_{p}\}=1/\alpha> 1.$$ 
Then 
$$\rho(A_R)=\rho(C)=1/\alpha> 1,$$ 
since $T_R\in S_q,\; 1< q \leq1/\alpha$, i.e., $\rho(T_R)\leq \rho(T)=1$ (see \cite[p.118]{Gohberg2}), the orders of the Hermitian components of operators $T$ do not exceed the order of $T$. Then  we have $\rho(A)=1/\alpha>1$. 
This contradicts the fact that $\rho(A)\leq 1$.

Now we consider the case $T_J\in S_q, \; 1<q\leq1/\alpha$. Then by Theorem 11.1 \cite[p.215]{Gohberg} we obtain 
$$A_R=C+T_R\in S_q, \; 1<q\leq1/\alpha,$$
 i.e., $\rho(A_R)=1$. Then $\rho(A)=1$. But by the condition of Theorem \ref{Th2.1} we have $\rho(A)= \rho(C)=1/\alpha>1$.
These contradictions prove that A is not Volterra.

Case 2.  We consider the case $1\leq \alpha<2,\; 1/2<1/\alpha\leq 1$, then $0<q\leq1/\alpha\leq 1$. 

a) When $0<q\leq 1/2$ we get the contradiction with the condition of Theorem \ref{Th2.1} by virtue of Theorem 11.5 \cite[p.220]{Gohberg}, similarly as in case 1,  a). 

b) In the case of $1/2<q\leq1/\alpha\leq 1$, by the condition of Theorem \ref{Th2.1} we have that $\rho(C)=1/\alpha > 1/2$. But, by virtue of Theorem 11.3  \cite[p.217]{Gohberg}, which states that $C\in S_q, \; 1/2<q\leq1/\alpha\leq 1$, we have that $\rho(C)=1/2$. Then we get the contradiction. Thus we proved A is not Volterra.
\end{proof}

\begin{definition}
\label{def2.2}
The function $L(x)\, (b\leq x<\infty;\, b>0)$ will be called the \textit{slowly varying} if it is positive and continuously differentiable, and if \cite[p.144]{Gohberg} 
\[
\lim_{x\rightarrow\infty} \frac{x L'(x)}{L(x)}=0.
\]
\end{definition}
Further, we will need the following properties of slowly varying functions:
\begin{enumerate}
\item[1)]
for each $\lambda>0$ we have
\begin{equation}\label{22.1}
\lim_{x\rightarrow\infty} \frac{L(\lambda x)}{L(x)}=1;
\end{equation}

\item[2)]
for any $\epsilon>0$, one can always find constants $C'_{\epsilon}$,  $C''_{\epsilon}$ for a given slowly varying function $L(x)$ such that
\[
C'_{\epsilon} x^{-\epsilon}<L(x)<C''_{\epsilon} x^{\epsilon} \; (1\leq x<\infty);
\]

\item[3)]
for any $l>0$ one has
\begin{equation}\label{22.2}
\lim_{x\rightarrow\infty} \frac{L(x+l)}{L(x)}=1.
\end{equation}

\end{enumerate}

As an example of a slowly varying function we can take any function
\[
L(x)=(\ln x)^{\beta}
\]
where $\beta$ is an arbitrary real number.

\begin{theorem}\label{Th2.3}
If in the representation $A=C+T$ of the compact operator $A$ in a Hilbert space $H$, the $C$ is a non-negative compact operator such that
\[
\lim_{n\rightarrow\infty} \frac{1}{L(n)} s_n(C)=a,  \quad \mbox{for}\;\;  0< a<\infty,\;\;{and} \quad T\in S_q, \;\; q>0,,
\]
then the operator $A$ cannot be Volterra.
\end{theorem}

\begin{proof}
By Corollary 2.2 \cite[p.29]{Gohberg}, if $C, T\in S_{\infty}$ are compact then
\[
s_{n+m-1}(C+T)\leq s_{n}(C)+s_{m}(T), \quad n, m=1, 2, \dots.
\]
Then
\[ 
\begin{split}
s_{2n+1}(A)&=s_{2n+1}(C+T) \leq s_{2n}(A)=s_{2n}(C+T)\leq s_{2n-1}(C+T)\\
& \leq s_{n}(C)+s_{n}(T), \quad\;\; n=1, 2, \dots.
\end{split}
\]
Thus
\[
\begin{split}
 \frac{1}{L(n)} s_{2n+1}(C+T)& \leq\frac{1}{L(n)} s_{2n}(C+T)\leq \frac{1}{L(n)} s_{n}(C)+\frac{1}{L(n)} s_{n}(T)\\
&=\frac{1}{L(n)} s_{n}(C)+\frac{1}{L(n)\cdot n^{\frac{1}{q}}} \cdot n^{\frac{1}{q}}s_{n}(T),
 \end{split}
\]
\[
\begin{split}
\Rightarrow \quad \frac{L(2n+1)}{L(n)L(2n+1)} s_{2n+1}(C+T)&\leq\frac{L(2n)}{L(n)L(2n)} s_{2n}(C+T)\\
&\leq \frac{1}{L(n)} s_{n}(C)+\frac{1}{L(n)\cdot n^{\frac{1}{q}}} \cdot n^{\frac{1}{q}}s_{n}(T).
\end{split}
\]
Then by \eqref{22.1}, \eqref{22.2} we have
\[
\lim_{n\rightarrow\infty} \frac{s_{n}(C+T)}{L(n)} \leq a.
\]
 On the other hand $C=A-T$ and $s_{2n}(C)\leq s_n(A)+s_n(T)$. 
Furthermore
\[ 
\begin{split}
s_{2n+1}(C) \leq s_{2n}(C) \leq s_{n}(A)+s_{n}(T), \quad\;\; n=1, 2, \dots.
\end{split}
\]
Thus
\[
\begin{split}
 \frac{1}{L(n)} s_{2n+1}(C)& \leq\frac{1}{L(n)} s_{2n}(C)\leq \frac{1}{L(n)} s_{n}(A)+\frac{1}{L(n)} s_{n}(T)\\
&=\frac{1}{L(n)} s_{n}(A)+\frac{1}{L(n)\cdot n^{\frac{1}{q}}} \cdot n^{\frac{1}{q}}s_{n}(T),
 \end{split}
\]
\[
\begin{split}
\Rightarrow \quad \frac{L(2n+1)}{L(n)L(2n+1)} s_{2n+1}(C)&\leq\frac{L(2n)}{L(n)L(2n)} s_{2n}(C)\\
&\leq \frac{1}{L(n)} s_{n}(A)+\frac{1}{L(n)\cdot n^{\frac{1}{q}}} \cdot n^{\frac{1}{q}}s_{n}(T).
\end{split}
\]
and we have
\[
a\leq \lim_{n\rightarrow\infty} \frac{1}{L(n)} s_n(A),
\]
then consequently
\[
\lim_{n\rightarrow\infty} \frac{1}{L(n)} s_n(A)=\lim_{n\rightarrow\infty} \frac{1}{L(n)} s_n(C+T)=a.
\]
Thus we have $A= C+T\notin S_p$, for all $p>0$.

To prove Theorem \ref{Th2.3} assume the contrary, i.e., suppose that the operator $A$ is  Volterra.

Consider the following cases

1) If $0<q\leq 1/2$ then by Theorem 11.5  \cite[p.220]{Gohberg} we get that 
$$\lim_{n\rightarrow\infty} n^2 s_n(C)$$ 
exists and is finite.
This contradicts the condition of Theorem \ref{Th2.3}.

2) If $1/2<q\leq 1$ then by Theorem 11.3 \cite[p.217]{Gohberg} we get   $C\in S_q, \;\; 1/2<q\leq 1$, but this contradicts with the condition of Theorem \ref{Th2.3}, i.e., with $C\notin S_p$, for all $p>0$.

3) If $1<q<\infty$ then we consider following cases:

a) If $T_J\in S_q, \;\; 1<q<\infty$ then by Theorem 11.1 \cite[p.215]{Gohberg} $C+T_R=A_R\in S_q,  \;\; 1<q<\infty \quad \Rightarrow \quad A\in S_q, \;\; 1<q<\infty.$ 
 But this contradicts with the condition of Theorem \ref{Th2.3}, i.e., with the $A =C+T\notin S_p$, for all $p>0$.

b) If $T_J\in S_q, \;\; 0<q\leq 1$ and $T_R\in S_q, \;\; 1<q<\infty$,  then  by the Theorem 2.1  \cite[p.83]{Gohberg2} we have $A\in S_{\Omega}$ which is of order $\rho(A)\leq 1$. But this contradicts with the condition of Theorem \ref{Th2.3}, i.e., with the $A =C+T\notin S_p$, for all $p>0$.

\end{proof}

\begin{corollary}\label{Cor2.4}
The results of Theorem \ref{Th2.1} and Theorem \ref{Th2.3}  are also true for the operator $A^*$.
\end{corollary}


\section{\large An application of Theorem \ref{Th2.1} and Theorem \ref{Th2.3} to the Laplace operator}
\label{section3}\label{sec3}

As a rule, the study of the spectral properties of operators generated by differential equations of hyperbolic or parabolic type with Cauchy initial data involves Volterra boundary well-defined problems. But Hadamard's example shows that the Cauchy problem for the Laplace equation is ill posed. At present, not a single Volterra well-defined restriction or extension for elliptic-type equations is known. Thus, the following question arises: Does there exist a Volterra well-defined restriction of the maximal operator $\widehat{L}$ or a Volterra well-defined extension of the minimal operator $L_0$ generated by the Laplace operator \eqref{3.1}? This section is devoted to the study of this question.

In the Hilbert space $L_2(\Omega)$, where $\Omega$ is a bounded domain in $\mathbb R^m$ with infinitely smooth boundary $\partial\Omega$, let us consider the minimal $L_0$ and maximal $\widehat{L}$ operators generated by the Laplace operator
\begin{equation}\label{3.1}
-\Delta u=-\biggl(\frac{\partial^2 u}{\partial{x_1^2}}+\frac{\partial^2 u}{\partial{x_2^2}}+\cdots+\frac{\partial^2 u}{\partial{x_m^2}}\biggr).
\end{equation}

The closure $L_0$ in the space $L_2(\Omega)$ of the Laplace operator~\eqref{3.1} with domain $C_0^\infty (\Omega)$ is called the minimal operator corresponding to the Laplace operator.

The operator $\widehat{L}$, adjoint to the minimal operator $L_0$ corresponding to the Laplace operator is called the maximal operator corresponding to the Laplace operator (see \cite{Hormander}). Note that 
$$D(\widehat{L})=\{u\in L_2(\Omega): \; \widehat{L}u =-\Delta u \in L_2(\Omega)\}.$$

Denote by $L_D$ the operator corresponding to the Dirichlet problem with the domain 
\[ D(L_D)=\{u\in W_2^2(\Omega): \; u|_{\partial\Omega}=0\}. \]

Then, in view of \eqref{1.2}, the inverse operators $L^{-1}_K$ to all possible well-defined restrictions of the maximal
operator $\widehat{L}$ corresponding to the Laplace operator~\eqref{3.1}  have the following form:
\begin{equation}
u\equiv L^{-1}_Kf=L_D^{-1}f+Kf, \label{3.2}
\end{equation}
 where, in view of \eqref{1.1}, $K$ is an arbitrary bounded (in $L_2(\Omega)$) operator with 
 $$
 R(K)\subset \mbox{Ker}\, \widehat{L} =\{u\in L_2(\Omega): \: -\Delta u=0\}.
 $$
Then the direct operator $L_K$ is determined from the following problem:
\begin{equation}
\widehat{L}u=f, \quad f\in L_2(\Omega),\label{3.3}
\end{equation}
\begin{equation}
D(L_K)=\{ u\in D(\widehat{L}) : \; (I-K\widehat{L})u|_{\partial\Omega}=0 \}, \label{3.4}
\end{equation}
where $I$ is the unit operator in $L_2(\Omega)$.
There are no other linear well-defined restrictions of the operator $\widehat{L}$  (see \cite{Biyarov}).

The operators $(L^*_K)^{-1}$, corresponding to the operators $L^*_K$,
\[
v\equiv (L^*_K)^{-1}g=L_D^{-1}g+K^*g,
\]
describe the inverse operators to all possible well-defined extensions of the minimal operator $L_0$ if and
only if $K$ satisfies the condition (see \cite{Biyarov}) 
\[ \mbox{Ker}\,(I+K^*L_D)=\{0\}. \]
Note that the last condition is equivalent to the following one: $\overline{D(L_K)}= L_2(\Omega)$.
If the additional condition 
\[ KR(L_0)=\{0\} \]
is imposed on the operator $K$ from \eqref{3.2}, then the operator $L_K$ corresponding to problem \eqref{3.3},  \eqref{3.4}, will turn out to be boundary well-defined.

In representation \eqref{3.2} we can always take any other positive well-defined restriction $L$ of the maximal operator $\widehat{L}$ instead of $L_D$, such that:
\[
lim_{n\rightarrow\infty} n^2 s_n(L^{-1})=\infty
\]
and
\[
\lim_{n\rightarrow\infty} n^{\alpha} s_n(L^{-1})=a>0, \quad \mbox{for} \quad 0<\alpha<2, \quad \mbox{or} \quad \lim_{n\rightarrow\infty} \frac{1}{L(n)} s_n(L^{-1})=a>0.
\]  
Note that the existence of such operator $L$ is achieved using the operator $K$ in \eqref{3.2}. 
Then the well-defined restriction of the maximal operator $\widehat{L}$ generated by \eqref{3.1} have the inverse operators of the following form
\begin{equation}\label{3.5}
u\equiv L^{-1}_K f=L^{-1}f+Kf,
\end{equation}
where $K$ is an arbitrary bounded operator (in $L^2(\Omega))$ with
\[
R(K)\subset \mbox{Ker}\, \widehat{L}, \quad D(L_K)=\big\{u\in D(\widehat{L}):\, (I-K\widehat{L})u\in D(L)\big\}.
\]

An application of Theorem \ref{Th2.1} and Theorem \ref{Th2.3} to the Laplace operator gives following theorems.


\begin{theorem}\label{Th3.1}
If the operator $L_K^{-1}$ in a Hilbert space $L_2(\Omega)$ can be represented in the form \eqref{3.5}, where $L^{-1}$ is a non-negative operator such that
\[
\lim_{n\rightarrow\infty} n^2 s_n(L^{-1})=\infty, \quad \lim_{n\rightarrow\infty} n^{\alpha} s_n(L^{-1})=a, \quad \mbox{for} \quad 0<\alpha<2, \;\;\; 0< a<\infty,
\]
and the operator $K\in S_q(L_2(\Omega))$ for all $q\leq\frac{1}{\alpha}$, then the operator $L_K^{-1}$ cannot be  Volterra, i.e., $L_K$ is the non Volterra well-defined restriction of the maximal operator $\widehat{L}$.
\end{theorem}

\begin{remark}\label{rem3.2}
In the special case when $L=L_D$, i.e., $\alpha=2/m$, this result was obtained by the author in \cite{Biyarov2}.
\end{remark}

\begin{theorem}\label{Th3.3}
If in the representation \eqref{3.5} of the operator $L_K^{-1}$ in a Hilbert space $L_2(\Omega)$, the $L^{-1}$ is a non-negative operator such that
\[
\lim_{n\rightarrow\infty} \frac{1}{L(n)} s_n(L^{-1})=a,\;\;\quad \mbox{for}\;\;  0< a<\infty,  \quad \mbox{and} \quad K\in S_q, \;\; q>0,
\]
then the operator $L_K^{-1}$ cannot be Volterra, i.e., $L_K$ is the non Volterra well-defined restriction of the maximal operator $\widehat{L}$.
\end{theorem}

\begin{remark}\label{rem3.4}
The results obtained in Theorem \ref{Th3.1} and Theorem \ref{Th3.3} are also valid for elliptic operators of order $2l$ ($l\in \mathbb{N}$):
\[
(\mathcal{L}u)(x)=\sum_{|\alpha|, |\beta|\leq l} (-1)^{|\alpha|+|\beta|} D^{\alpha}(a_{\alpha\beta}(x)D^{\beta}u), \quad x\in \Omega,
\]
where $a_{\alpha\beta}\in C^l(\Omega)$ are the real-valued functions for all multi-indices $\alpha, \beta$, and  $\Omega$ is a bounded domain in $\mathbb{R}^m$ with infinitely smooth boundary $\partial \Omega$.
We assume the uniform ellipticity condition is satisfied, i.e., for some $\nu>0$
\[
\sum_{|\alpha|=|\beta|=l}a_{\alpha\beta}(x)\xi^{\alpha}\xi^{\beta}\geq \nu |\xi|^{2l}, \quad \mbox{for all} \;\; x, \xi \in \Omega.
\]
\end{remark}

\begin{corollary}\label{Cor3.5}
The results of Theorem \ref{Th3.1} and Theorem \ref{Th3.3}  are also true for the operator $L_K^*$.
\end{corollary}


\section{\large On the absence of Volterra well-defined restrictions and extensions of the Laplace operator in the case $m=2$}
\label{section4}\label{sec4}

In Sections \ref{sec2} and \ref{sec3}, we found a sufficient conditions for non-Volterra operators of one perturbation class of positive operators, i.e., we proved the existence of at least one nonzero eigenvalue. In this section, we prove that for the Laplace operator in the two-dimensional case there are no Volterra well-defined restrictions or extensions at all.

We pass to the polar coordinate system:
\[x=r \cos\varphi, \: y=r \sin\varphi.\]
Then the Laplace operator \eqref{3.1} has the form
\begin{equation} 
\label{4.1} 
\widehat{L}u \equiv -\Delta u=-\frac{\partial^2 u}{\partial{x^2}}-\frac{\partial^2 u}{\partial{y^2}}=-u_{rr}-\frac{1}{r}u_r-\frac{1}{r^2}u_{\varphi\varphi}=f(r, \varphi), 
\end{equation}
on
\[D(\widehat{L})=\{u\in L_2(\Omega): \; \Delta u \in L_2(\Omega)\},\]
is the maximal operator (see \cite{Biyarov2}), where $\Omega$ is unit circle.
In this case denote by $L_D$ the operator corresponding to the Dirichlet problem with the domain 
\[ D(L_D)=\{u\in W_2^2(\Omega): \; u|_{r=1}=0\}. \]
 Then any well-defined restriction $L_K$ from \eqref{3.2}  acts as the maximal operator $\widehat{L}$ on the domain
\begin{equation} 
\label{4.2} 
D(L_K)=\{u\in D(\widehat{L}): \; [(I-K \widehat{L})u]|_{r=1}=0  \}, 
\end{equation}
where $K$ is any bounded linear operator in $L_2(\Omega)$ such that
$R(K)\subset \mbox{Ker}\,\widehat{L}$.
Compactness of $L^{-1}_K$ is necessary condition for $L_K$ to be Volterra. 
From \eqref{3.2} note that $L^{-1}_K$ is compact if and only if $K$ is a compact operator, since $L^{-1}_D$ is compact. Then for $K$ the Schmidt expansion takes place (see \cite[p.\,47(28)]{Gohberg})
\begin{equation}\label{4.3}
K=\sum_{j=1}^{\infty}s_j(\:\cdot\,, \,Q_j)F_j               
\end{equation}       
where $\{Q_j \}_1^\infty$ is orthonormal system in $L_2(\Omega)$, $\,\{F_j \}_1^\infty$ is orthonormal system in $\mbox{Ker}\,\widehat{L}\,$, and $\,\{s_j \}_1^\infty$ is a monotone sequence of non-negative numbers converging to zero. The series on the right side of \eqref{4.3} converges in the uniform operator norm. We now state the main result of this paper.
\begin{theorem}\label{Theorem4.1}
Let $\widehat{L}$ be a maximal operator generated by the Laplace~\eqref{3.1} in $L_2(\Omega)$ when $m=2$. Then any well-defined restriction $L_K$ of the maximal operator $\widehat{L}$, with the compact operator $K$ in representation \eqref{3.2}, cannot be Volterra, i.e., the problem \eqref{4.1} and \eqref{4.2} cannot be Volterra. 
\end{theorem}
\begin{proof}
Let us prove by contradiction. Suppose that there exists a Volterra well-defined restriction $L_K$. This is equivalent to the existence of a such compact operator $K$ that the operator $L_K$ has no nonzero eigenvalue. The general solution of the equation
\[\widehat{L}u=-\Delta u=-u_{rr}-\frac{1}{r}u_r-\frac{1}{r^2}u_{\varphi\varphi}=\lambda^2u,\]
from the space $L_2(\Omega)$ has the form (see \cite{Vekya}) 
\[u(r,\varphi)=u_0(r,\varphi)-\int_0^r u_0(\rho,\varphi)\frac{\partial}{\partial\rho} J_0(\lambda\sqrt{r(r-\rho)})\,d\rho,\]
where $\lambda$ is any complex number, $u_0(r,\varphi)$ is the solution of the equation
\[\widehat{L}u_0\equiv-\Delta u_0=-u_{0_{rr}}-\frac{1}{r}u_{0_r}-\frac{1}{r^2}u_{0_{\varphi\varphi}}=0,\]
which is a harmonic function from the space $L_2(\Omega)$ and 
\[J_0(z)=\sum_{n=0}^\infty\dfrac{(-1)^n}{(n!)^2}\Bigl(\dfrac{z}{2}\Bigr)^{2n}\]  
is the Bessel function.
Then, by virtue of \eqref{4.2} we obtain the equation
\begin{equation}\label{4.4}
\begin{split}
&u_0(1,\varphi)-\int_0^1 u_0(\rho,\varphi)\frac{\partial}{\partial\rho}J_0(\lambda\sqrt{1-\rho})\,d\rho \\[5pt]
&\, -\lambda^2 \sum_{j=1}^{\infty}s_j F_j(1, \varphi)\cdot \frac{1}{2\pi}\int_0^{2\pi}\int_0^1u_0(\rho,\theta)\cdot \overline{Q_j(\rho,\theta)} \rho\, d\rho\, d\theta \\[5pt]
&\, +\lambda^2 \sum_{j=1}^{\infty}s_j F_j(1, \varphi)\frac{1}{2\pi}\int_0^{2\pi}\int_0^1 \overline{Q_j(\rho,\theta)}\cdot \int_0^\rho u_0(\tau,\theta)\frac{\partial}{\partial\tau}J_0(\lambda \sqrt{\rho(\rho-\tau)}\,d\tau\rho \,d\rho\, d\theta =0.
\end{split}
\end{equation}

The considered problem on the spectrum for the Laplace operator has no eigenvalues if and only if the equation \eqref{4.4} has no zeros as a function of $\lambda$. The harmonic function $u_0(\rho,\varphi)$ does not depend on $\lambda$. It is easy to notice that the left side of the equation is an entire function of exponential type not higher than the first order. Then, by virtue of Picard's theorem (see \cite[p.\,264, 266]{Markushevich}), this function has the form $C e^{d\lambda}$, where $C(\varphi)$ and $d(\varphi)$ are a functions which are independent of $\lambda$. If we notice that the left side of the equation \eqref{4.4} is even with respect to the sign of $\lambda$, then $d=0$. By equating these functions when $\lambda=0$ we have $C=u_0(1,\varphi)$. Then we get the following
\begin{equation}\label{4.5}
\begin{split}
&-\int_0^1 u_0(\rho,\varphi)\frac{\partial}{\partial\rho}J_0(\lambda\sqrt{1-\rho})\,d\rho \\[5pt]
&\, -\lambda^2 \sum_{j=1}^{\infty}s_j F_j(1, \varphi)\cdot \frac{1}{2\pi}\int_0^{2\pi}\int_0^1u_0(\rho,\theta)\cdot \overline{Q_j(\rho,\theta)} \rho\, d\rho\, d\theta \\[5pt]
&\;\; +\lambda^2 \sum_{j=1}^{\infty}s_j F_j(1, \varphi)\cdot \frac{1}{2\pi}\int_0^{2\pi}\int_0^1 \overline{Q_j(\rho,\theta)} \int_0^\rho  u_0(\tau,\theta)\frac{\partial}{\partial\tau}J_0(\lambda \sqrt{\rho(\rho-\tau)}\,d\tau\,\rho\, d\rho\, d\theta =0.
\end{split}
\end{equation}
Divide both sides of \eqref{4.5} by $\lambda^2$ and let $\lambda$ tend to zero. Then
\begin{equation}\label{4.6}
\sum_{j=1}^{\infty}s_j F_j(1, \varphi)\cdot \frac{1}{2\pi}\int_0^{2\pi}\int_0^1 u_0(\rho,\theta) \overline{Q_j(\rho,\theta)}\rho \,d\rho\, d\theta= -\frac{1}{4}\int_0^1 u_0(\rho,\varphi)\,d\rho.
\end{equation}
Under the condition that \eqref{4.6} is fulfilled we obtain
\begin{equation}\label{4.7}
\begin{split}
&-\int_0^1 u_0(\rho,\varphi)\biggl[\frac{\partial J_0}{\partial\rho}(\lambda\sqrt{1-\rho})+\frac{1}{4}\biggr]d\rho \\[5pt]
&\, + \sum_{j=1}^{\infty}s_j F_j(1, \varphi)\cdot \frac{1}{2\pi}\int_0^{2\pi}\int_0^1 \overline{Q_j(\rho,\theta)} \int_0^\rho  u_0(\tau,\theta)\frac{\partial}{\partial\tau} J_0(\lambda \sqrt{\rho(\rho-\tau)}\,d\tau\,\rho\, d\rho\, d\theta =0.
\end{split}
\end{equation}
On the left side of the equation \eqref{4.7} we make the change of variables: in the first summand $t=\sqrt{1-\rho}$,  in the second summand $t=\sqrt{\rho(\rho-\tau)}$. Then we have
\begin{equation}\label{4.8}
\begin{split}
&\int_0^1 u_0(1-t^2,\varphi)\biggl[\frac{J_0'(\lambda t)}{2\lambda t}+\frac{1}{4}\biggr] 2t\,dt \\[5pt]
&\, - \sum_{j=1}^{\infty}s_j F_j(1, \varphi)\cdot \frac{1}{2\pi}\int_0^{2\pi}\int_0^1 J_0'(\lambda t) \int_t^1 u_0\Bigl(\frac{\rho^2-t^2}{\rho}, \theta\Bigr) \overline{Q_j(\rho,\theta)} \rho\, d\rho\, dt\, d\theta =0.
\end{split}
\end{equation}
For the Bessel function the following equalities hold
\[\dfrac{J'_0(\lambda t)}{2\lambda t}+\dfrac{1}{4}=\dfrac{1}{4}\sum_{n=1}^\infty \dfrac{(-1)^{n+1}}{((n+1)!)^2}\Bigl(\dfrac{\lambda t}{2}\Bigr)^{2n}\cdot(n+1),\]
and
\[\lambda J'_0(\lambda t)=\sum_{n=1}^\infty \dfrac{(-1)^{n}}{(n!)^2}\Bigl(\dfrac{\lambda t}{2}\Bigr)^{2n}\cdot n\cdot\dfrac{2}{t}.\]
Substitute them into \eqref{4.8} and equate the coefficients of $\lambda^{2n}$ to zero
\[
\begin{split}
&\int_0^1 u_0(1-t^2,\varphi)\frac{-1}{4(n+1)}\cdot t^{2n}\cdot 2t\,dt \\[5pt]
&\, - \sum_{j=1}^{\infty}s_j F_j(1, \varphi)\cdot \frac{1}{2\pi}\int_0^{2\pi}\int_0^1 nt^{2n}\cdot \frac{2}{t} \int_t^1 u_0\Bigl(\frac{\rho^2-t^2}{\rho}, \theta\Bigr) \overline{Q_j(\rho,\theta)} \rho\, d\rho\, dt\, d\theta =0.
\end{split}
\]
We make the transformation of the following type
\[\frac{1}{n+1}\cdot t^{2n}=\frac{2}{t^2}\int_0^t\tau^{2n+1}d\tau,\qquad n\cdot t^{2n}=\frac{t}{2}\cdot \frac{\partial}{\partial t}\Bigl(t^{2n}\Bigr).\]
Then
\[\begin{split}
&\int_0^1t^{2n}\cdot \biggl\{ t\int_t^1 u_0(1-\tau^2,\varphi) \frac{d\tau}{\tau}\,dt\\[5pt]
&\quad-\frac{1}{2\pi}\int_0^{2\pi} \frac{\partial}{\partial t} \sum_{j=1}^{\infty}s_j F_j(1, \varphi)\cdot \int_t^1 u_0\Bigl(\frac{\rho^2-t^2}{\rho}, \theta\Bigr) \overline{Q_j(\rho,\theta)} \rho\, d\rho\, d\theta \biggr\}\, dt=0.
\end{split}\]
In view of the completeness of the system of functions $\{t^{2n}\}_1^\infty$ in $L_2(0,1)$ we obtain (see \cite[p.\,107]{Kach})
\[
 t\int_t^1 u_0(1-\tau^2,\varphi) \frac{d\tau}{\tau}
-\frac{1}{2\pi}\int_0^{2\pi} \frac{\partial}{\partial t} \sum_{j=1}^{\infty}s_j F_j(1, \varphi)\cdot \int_t^1 u_0\Bigl(\frac{\rho^2-t^2}{\rho}, \theta\Bigr) \overline{Q_j(\rho,\theta)} \rho\, d\rho\, d\theta=0.
\]
Integrating this equation from $t$ to $1$, we get
\begin{equation}\label{4.9}
\begin{split}
&\,\int_t^1 u_0(1-\tau^2,\varphi) \frac{\tau^2-t^2}{2\tau}\,d\tau\\[5pt]
&\quad+\frac{1}{2\pi}\int_0^{2\pi} \int_t^1 u_0\Bigl(\frac{\rho^2-t^2}{\rho}, \theta\Bigr) \sum_{j=1}^{\infty}s_j F_j(1, \varphi) \overline{Q_j(\rho,\theta)} \rho\, d\rho\, d\theta=0.
\end{split}
\end{equation}
where $\,0\leq t\leq1, \;\, 0\leq\varphi<2\pi$. Note that the condition \eqref{4.9} contains the condition \eqref{4.6} as a particular case when $t=0$. Condition \eqref{4.9} will turn out to be the Volterra criterion of the well-defined restriction $L_K$, if it holds for any harmonic function $u_0(r, \varphi)$ from $L_2(\Omega)$.

By Poisson's formula
\[u_0(r,\varphi)=\frac{1}{2\pi}\int_0^{2\pi}\dfrac{1-r^2}{1-2r\cos(\varphi-\gamma)+r^2}u_0(1,\gamma)\,d\gamma,\]
the equality \eqref{4.9} is transformed to
\[
\begin{split}
&\frac{1}{2\pi}\int_0^{2\pi}u_0(1,\gamma)\biggl\{\int_t^1\dfrac{1-(1-\tau^2)^2}{1-2(1-\tau^2)\cos(\varphi-\gamma)+(1-\tau^2)^2}\cdot \dfrac{\tau^2-t^2}{2\tau}\,d\tau\\[5pt]
&\,+\frac{1}{2\pi}\int_0^{2\pi} \int_t^1\dfrac{1-\bigl(\frac{\rho^2-t^2}{\rho}\bigr)^2}{1-2\bigl(\frac{\rho^2-t^2}{\rho}\bigr)\cos(\theta-\gamma)+\bigl(\frac{\rho^2-t^2}{\rho}\bigr)^2}\sum_{j=1}^{\infty}s_j F_j(1, \varphi) \overline{Q_j(\rho,\theta)}\rho\, d\rho\, d\theta\biggr\}\,d\gamma=0.
\end{split}
\]
Taking into account the density of the set of functions $u_0(1, \varphi)$ in $L_2(0, 2\pi)$ for almost all values of $t \, (0\leq t\leq1), \; \varphi \; (0\leq\varphi<2\pi)$ we obtain the equality
\begin{equation}\label{4.10}
\begin{split}
&\int_t^1\dfrac{1-(1-\tau^2)^2}{1-2(1-\tau^2)\cos(\varphi-\gamma)+(1-\tau^2)^2}\cdot \dfrac{\tau^2-t^2}{2\tau}\,d\tau\\[5pt]
&\,+\frac{1}{2\pi}\int_0^{2\pi} \int_t^1\dfrac{1-\bigl(\frac{\rho^2-t^2}{\rho}\bigr)^2}{1-2\bigl(\frac{\rho^2-t^2}{\rho}\bigr)\cos(\theta-\gamma)+\bigl(\frac{\rho^2-t^2}{\rho}\bigr)^2}\sum_{j=1}^{\infty}s_j F_j(1, \varphi) \overline{Q_j(\rho,\theta)}\rho\, d\rho\, d\theta=0.
\end{split}
\end{equation}
Now the equation \eqref{4.10} is the Volterra criterion of the well-defined restriction $L_K$ of the maximal operator $\widehat{L}$ generated by the Laplace operator \eqref{3.1} in $L_2(\Omega)$, where $\Omega$ is the unit disk. 

Further, we apply to the equation \eqref{4.10} the Poisson operator of the variables $r$ and $\varphi$. The first summand we transform with the formula of the superposition of two Poisson integrals (see {\cite[p. 140]{Axler}}), and in the second summand the harmonic function $F_j(r,\varphi)$ is reproduced by Poisson's formula. We have
\[
\begin{split}
&\int_t^1\dfrac{1-r^2(1-\tau^2)^2}{1-2r(1-\tau^2)\cos(\varphi-\gamma)+r^2(1-\tau^2)^2}\cdot \dfrac{\tau^2-t^2}{2\tau}\,d\tau\\[5pt]
&\,+\frac{1}{2\pi}\int_0^{2\pi} \int_t^1\dfrac{1-\bigl(\frac{\rho^2-t^2}{\rho}\bigr)^2}{1-2\bigl(\frac{\rho^2-t^2}{\rho}\bigr)\cos(\theta-\gamma)+\bigl(\frac{\rho^2-t^2}{\rho}\bigr)^2}\sum_{j=1}^{\infty}s_j F_j(r, \varphi) \overline{Q_j(\rho,\theta)}\rho\, d\rho\, d\theta=0.
\end{split}
\]
From this equality by using the orthonormality of the system $\{F_j(r,\varphi)\}_1^\infty$ we obtain the relation between the orthonormal systems $\{F_j\}_1^\infty$ and $\{Q_j\}_1^\infty$ of the following form
\begin{equation}\label{4.11}
\begin{split}
&\int_t^1 \biggl\{\frac{1}{2\pi}\int_0^{2\pi}\int_0^1\dfrac{1-r^2(1-\tau^2)^2}{1-2r(1-\tau^2)\cos(\varphi-\gamma)+r^2(1-\tau^2)^2} \overline{F_j(r,\varphi)}r\, dr\,d\varphi\biggr\}\cdot \dfrac{\tau^2-t^2}{2\tau}\,d\tau\\[5pt]
&\,=-\frac{1}{2\pi}\int_0^{2\pi} \int_t^1\dfrac{1-\bigl(\frac{\rho^2-t^2}{\rho}\bigr)^2}{1-2\bigl(\frac{\rho^2-t^2}{\rho}\bigr)\cos(\theta-\gamma)+\bigl(\frac{\rho^2-t^2}{\rho}\bigr)^2}\cdot s_j \overline{Q_j(\rho,\theta)}\rho\, d\rho\, d\theta, \quad j=1, 2, \ldots
\end{split}
\end{equation}
In both parts of the equality \eqref{4.11} we use the expansion of the Poisson kernel
\[
\frac{1-r^2}{1-2r\cos \varphi+r^2}=1+2\sum_{n=1}^\infty r^n\cos n\varphi.
\]
We obtain the equality of the two Fourier series in the orthogonal system $\{1/2, \cos  \gamma, \sin  \gamma, \ldots,$ $\cos n\gamma, \sin n\gamma, \ldots\}$ in $L_2(0, 2\pi)$. Equating the coefficients, we get the following system of equations
\begin{equation}\label{4.12}
\left \{\begin{split}
&\frac{1}{2\pi}\int_0^{2\pi}\int_0^1 \overline{F_j(r,\varphi)}\cdot r\,dr\,d\varphi\cdot \int_t^1\dfrac{\tau^2-t^2}{2\tau}\,d\tau=-\frac{1}{2\pi}\int_0^{2\pi}\int_t^1 s_j \overline{Q_j(\rho,\theta)}\rho\, d\rho\, d\theta, \\[8pt]
&\frac{1}{\pi}\int_0^{2\pi}\int_0^1 \overline{F_j(r,\varphi)}\cdot r^{n+1}\cos n\varphi\, dr\,d\varphi\int_t^1 (1-\tau^2)^n\dfrac{\tau^2-t^2}{2\tau}\,d\tau \\[4pt]
&\quad=-\int_t^1\frac{1}{\pi}\int_0^{2\pi}s_j \overline{Q_j(\rho,\theta)}\cdot \cos n\theta\, d\theta \biggl(\dfrac{\rho^2-t^2}{\rho}\biggr)^n\rho\, d\rho ,\\[8pt]
&\frac{1}{\pi}\int_0^{2\pi}\int_0^1 \overline{F_j(r,\varphi)}\cdot r^{n+1}\cdot\sin n\varphi\, dr\,d\varphi\int_t^1 (1-\tau^2)^n\dfrac{\tau^2-t^2}{2\tau}\,d\tau \\[4pt]
&\quad=-\int_t^1\frac{1}{\pi}\int_0^{2\pi}s_j \overline{Q_j(\rho,\theta)}\cdot \sin n\theta \,d\theta \biggl(\dfrac{\rho^2-t^2}{\rho}\biggr)^n\rho\, d\rho, \quad j=1, 2, \ldots, \quad n=1, 2, \ldots.
\end{split} \right.
\end{equation}
We denote
\[
A_{nj}=\frac{1}{\pi}\int_0^{2\pi}\int_0^1 \overline{F_j(r,\varphi)}\cos n\varphi\cdot r^n\cdot r\,dr\,d\varphi,\quad j=1, 2, \ldots, \quad n=0, 1, 2, \ldots,
\]
\[
B_{nj}=\frac{1}{\pi}\int_0^{2\pi}\int_0^1 \overline{F_j(r,\varphi)}\sin n\varphi\cdot r^n\cdot r\,dr\,d\varphi,\quad j=1, 2, \ldots, \quad n=1, 2, \ldots.
\]
From the first equation of the system \eqref{4.12} it is easy to find that
\[
\frac{1}{2\pi}\int_0^{2\pi}s_j \overline{Q_j(t,\theta)}d\theta=\frac{1}{2}A_{0j}\cdot \ln t.
\]
If we denote by
\[
\omega_{nj}(\rho)=\frac{1}{\pi}\int_0^{2\pi}s_j \overline{Q_j(\rho,\theta)}\cos n\theta\, d\theta\dfrac{1}{A_{nj}\rho^n},
\]
then the second equation transforms to
\begin{equation}\label{4.13}
\int_t^1(1-\tau^2)^n\dfrac{\tau^2-t^2}{2\tau}\,d\tau=-\int_t^1(\rho^2-t^2)^n\omega_{nj}(\rho)\rho\, d\rho,\quad n=1,2,\ldots, \quad j=1, 2, \ldots
\end{equation}
The third equation is transformed into the same equation \eqref{4.13}, if we denote by
\[
\omega_{nj}(\rho)=\frac{1}{\pi}\int_0^{2\pi}s_j \overline{Q_j(\rho,\theta)}\sin n\theta\, d\theta\,\dfrac{1}{B_{nj}\rho^n}.
\]
We solve the equation \eqref{4.13} with respect to $\omega_{nj}(\rho)$. Note that
\[\omega_{1j}(t)=-\dfrac{1-t^2}{2t^2},\quad \omega_{2j}(t)=-\dfrac{1-t^4}{4t^4};\]
Further, we get the recurrence relation
\[
\begin{split}
&(1-t^2)^{n-k}=-\int_t^1(\rho^2-t^2)^{n-k-2}\cdot(n-k)(n-k-1)\cdot 4t^2\omega_{nj}(\rho)\rho\, d\rho\\[5pt]
&\,+k\int_t^1(\rho^2-t^2)^{n-k-1}\cdot(n-k)\cdot 4\omega_{nj}(\rho)\rho\, d\rho, \quad n=2,3,4,\ldots, \quad k=0,1,2,\ldots,n-2.
\end{split}
\]
This relation is equivalent to the Cauchy problem
\[\omega_{nj}'(t)+\frac{2n}{t}\omega_{nj}(t)=\frac{1}{t}, \quad \omega_{nj}(1)=0.\]
By solving this problem we get
\[\omega_{nj}(t)=\frac{1-t^{-2n}}{2n}, \quad j=1, 2, \ldots.\]
Now we have the following relations between the orthonormal systems 
 $\{Q_j\}_1^\infty$ and $\{F_j\}_1^\infty$:

\begin{equation}\label{4.14}
\left \{\begin{split}
&\frac{1}{2\pi}\int_0^{2\pi} s_j \overline{Q_j(t,\theta)}\,d\theta=\frac{1}{2}A_{0j}\cdot \ln t
,  \quad n=0, \,\,  j=1, 2, \ldots\\[5pt]
&\frac{1}{\pi}\int_0^{2\pi}s_j \overline{Q_j(t,\theta)}\cdot \cos n\theta\, d\theta =-A_{nj}\cdot\dfrac{t^n-t^{-n}}{2n} ,\\[5pt]
&\frac{1}{\pi}\int_0^{2\pi}s_j \overline{Q_j(t,\theta)}\cdot \sin n\theta\, d\theta =-B_{nj}\cdot\dfrac{t^n-t^{-n}}{2n}, \quad n=1,2,\ldots, \quad j=1, 2, \ldots.
\end{split} \right.
\end{equation} 
Satisfying the Volterra criterion \eqref{4.10}, we obtained the relation \eqref{4.14}. By assumption $Q_j(t,\theta)$ is in $L_2(\Omega)$. Then the integral with respect to $t$ on the left-hand sides of the system of equations \eqref{4.14} exists. However, for an arbitrary orthonormal system $\{F_j\}_1^\infty$, for $n=1,2,\ldots$, the integral on the right-hand sides of the system of equations \eqref{4.14} with respect to $t$ from $0$ to $1$ does not exist. This means that there are no orthonormal systems $\{F_j\}_1^\infty$ and $\{Q_j\}_1^\infty$ satisfying the equality \eqref{4.10}. This contradicts our assumption that there exists a Volterra well-defined restriction $L_K$. Thus, Theorem \ref{Theorem4.1} is proved.
\end{proof}

\begin{corollary}\label{cor4.2} 
There does not exist a Volterra well-defined extension $L_K$ of the minimal operator $L_0$ generated by the Laplace operator \eqref{4.1} in a Hilbert space $L_2(\Omega)$, where $\Omega$ is the unit disk.
\end{corollary}
\begin{proof} Suppose that there exists a Volterra well-defined extension $L_K$ of the minimal operator $L_0$. From $L_0\subset L_K$ it follows that $L^*_K\subset L_0^* =\widehat{L}$. The adjoint of a Volterra operator is a Volterra operator. Then we get a contradiction to Theorem \ref{Theorem4.1}. This completes the proof of Corollary~\ref{cor4.2}.
\end{proof}

It was also noticed that in the case $m=1$ there exists many Volterra well-defined restrictins and extensions.
Indeed,
if in \eqref{4.11} the function $u(r,\varphi)$ does not depend on the angle $\varphi$,  then we get one-dimensional equation 
\[
\widehat{L}u \equiv -\Delta u=-u_{rr}-\frac{1}{r}u_r-=f(r), 
\]
in the weighted space $L_2 (r;0,1)$ with weight $r$.
Then the Volterra criterion \eqref{4.10} and the equation \eqref{4.14} determine the operator $K$ of the following form
\[Kf=\int_0^1 f(t)\ln t\cdot t\,dt.\]
$K$ corresponds to the well-defined restriction $L$ with domain 
$$D(L_K)= \{u\in  W_2^1 (r;\,0,1) : u(0)=0\}.$$ 
Then the well-defined restriction $L_K$ is Volterra, because its inverse operator 
\[u(r)=L^{-1}_Kf=\int_0^r\ln \frac{t}{r}f(t)t\,dt,\]
is Volterra in space $L_2 (r;\,0,1)$.

\begin{remark}\label{rem4.3}
Theorem \ref{Theorem4.1} is true for every bounded simply connected domain in the plane, for which the Dirichlet problem is well-defined and there exists a conformal mapping onto the unit disk.
\end{remark}

\begin{remark}\label{rem4.4}
The generalization of Theorem \ref{Theorem4.1} to the $m$-dimensional ball (where $m\geqslant~3$) does not cause problems but it is rather cumbersome to write down.
\end{remark}

\vskip 1 cm




\vskip 1 cm \footnotesize
\begin{flushleft}
   Bazarkan Nuroldinovich Biyarov, \\
   Institute of Mathematics and Mathematical Modeling,\\
   National Academy of Science Republic of Kazakhstan, \\
   28 Shevchenko str.,\\
   050010 Almaty, Kazakhstan\\
   E-mail: bbiyarov@gmail.com
\end{flushleft} 

\end{document}